\documentclass[11pt,reqno,a4paper]{amsart}
\usepackage{amsmath,amscd,amstext,amsthm,amsfonts,latexsym}
\usepackage[dvips]{graphicx}
\usepackage{graphicx, graphics}

\usepackage{amsmath,amssymb,amsthm}
\usepackage{color}

\theoremstyle{plain}
\newtheorem{theorem}{Theorem}[section]
\newtheorem{proposition}[theorem]{Proposition}
\newtheorem{lemma}[theorem]{Lemma}

\newtheorem{maincorollary}{Corollary}

\theoremstyle{definition}
\newtheorem{remark}[theorem]{Remark}

\newcommand{\vep}{\varepsilon}

\newcommand{\la} {\lambda}

\newcommand{\cF}{{\mathcal F}}
\newcommand{\cG}{{\mathcal G}}

\newcommand{\cU}{{\mathcal U}}

\newcommand{\cW}{{\mathcal W}}

\newcommand{\N}{\mathbb{Z}_+}

\newcommand{\diam}{\operatorname{{diam}}}

\title[Stability and limit theorems for compositions of hyperbolic dynamical systems]
{Stability and limit theorems for sequences \\ of uniformly  hyperbolic dynamics}

\begin{document}

\author{A. Castro and F. Rodrigues and P. Varandas}

\address{Armando Castro $\&$ Paulo Varandas, Departamento de Matem\'atica, Universidade Federal da Bahia\\
  Av. Ademar de Barros s/n, 40170-110 Salvador, Brazil.}
\email{aacastro@ufba.br}
\email{paulo.varandas@ufba.br}

\address{Fagner B. Rodrigues, Departamento de Matem\'atica, Universidade Federal do Rio Grande do Sul, Brazil. \& CMUP, University of Porto, Portugal}
\email{fagnerbernardini@gmail.com}

\begin{abstract}
In this paper we obtain an almost sure invariance principle for
convergent sequences of either Anosov diffeomorphisms
or expanding maps on compact Riemannian manifolds and prove an ergodic stability result for such sequences.
The sequences of maps need not correspond to typical points of a
random dynamical system.
The methods in the proof rely on the stability of compositions of hyperbolic dynamical systems. We introduce
the notion of sequential conjugacies
and prove that these vary in a Lipschitz way with respect to the generating sequences of
dynamical systems.
As a consequence, we prove stability results for time-dependent expanding maps that complement
results in~\cite{Franks74} on time-dependent Anosov diffeomorphisms.
 \end{abstract}



\date{\today}
\maketitle

\section{Introduction}

Given a measurable map $f: X\to X$ and an $f$-invariant and ergodic probability measure $\mu$, the celebrated Birkhoff's ergodic
theorem assures that for every $\phi\in L^1(\mu)$, the C\'esaro averages $\frac1n \sum_{j=0}^{n-1}\phi(f^j(x))$ are almost everywhere
convergent to $\int \phi\, d\mu$.
Although the random variables $\{\phi\circ f^j\}$ are identically distributed, in general these fail to be independent.  Nevertheless,
the classical results as the central limit theorem and almost sure invariance principles hold for dynamics
with some hyperbolicity (see e.g. \cite{DP,Viana2} and references therein).

In the last decade much effort has been done in order to extend the classical limit theorems for this non-stationary context, namely for the compositions of uniformly expanding maps and piecewise  expanding interval maps (see \cite{Aimino, CR,DS,HNTV,NTV} and
references therein).
In this context the natural random variables obtained by the sequential dynamics are neither independent nor stationary
The strategy used in the large majority of these contributions is to describe the limit properties of non-stationary compositions of Perron-Frobenius operators and to provide limit theorems under mild assumptions on the growth of variances of the appropriate random
variables. Other relevant contributions include also the study of the fast loss of memory on the compositions of Anosov diffeomorphisms ~\cite{Stenlund} and the robustness of ergodic properties for compositions of piecewise expanding maps obtained in \cite{TPS}.

The notion of sequential dynamical systems was introduced in \cite{BB}. The dynamics of these non-autonomous dynamical system present substantial differences
in comparison with the classical dynamical systems context. In order to illustrate some difficulties one would like to mention that the
non-wandering set for sequential dynamical systems is compact set but in general it fails to be invariant by
the sequence of dynamics.
The study of sequential dynamics and the problem of its stability is motivated by the adaptative behavior of biological phenomena
(see e.g. \cite{DS} and references therein).

Our two goals here are to provide limit theorems for convergent sequences of hyperbolic dynamical systems and to prove
a time-dependent stability for sequences of dynamical systems.
More precisely, we consider sequences $\cF=\{f_n\}_n$ of Anosov diffeomorphisms (or expanding maps) such that $d_{C^1}(f_n,f)$ tends
to zero as $n$ tends to infinity. The usual strategy to prove limit theorems for sequential dynamics uses
the Perron-Frobenius transfer operator to show that such dynamics can be well approximated by reverse martingales.
In a context of convergent sequences of maps, our approach is substantially different. We prove that $C^1$-close sequential dynamics
formed by either Anosov diffeomorphisms or expanding maps are stable: there exists a sequence of homeomorphisms that conjugate
the dynamics (cf. Theorem~\ref{cor:Franks}). Moreover, these sequential conjugacies vary in a Lipschitz way with respect to
the sequences of dynamics (see Theorem~\ref{thm:Ainvertivel}). In the case of convergent sequences the sequence of
homeomorphisms is $C^0$-convergent to the identity. This allows us to prove an ergodic stability for hyperbolic maps, roughly meaning
that the behavior of Birkhoff averages for continuous observables is similar to the one observed by convergent sequences of
nearby dynamics. We refer the reader to Theorem~\ref{thm:asymptinv} for precise statements. Since we obtain quantitative results
for the velocity of the previous convergence to identity, limit theorems such as the central limit theorem, the functional central limit theorem and the law of the iterated logarithm transfer from the Brownian motion to time-series generated by observations on the
sequential dynamical system.

\section{Statement of the main results}\label{sec:statements}

A sequential dynamical system is
given by a collection $\mathcal F=\{f_n\}_{n\in \mathcal Z}$ of continuous maps $f_n : X_n \to X_{n+1}$, where each
$(X_n,d_n)$ is a complete metric space for every $n\in \mathcal Z$ and $\mathcal Z=\mathbb Z^+$ or $\mathbb Z$.
We endow the space of sequences of
maps with the $C^r$ topology. More precisely,
fix $r\ge 0$ and let $\mathcal S^r ((X_n)_n)$ denote the space sequences of $C^r$-differentiable maps
$\mathcal F=\{f_n\}_{n\in \mathbb Z}$, where $f_n : X_n \to X_{n+1}$ and every $X_n$ is a
smooth manifold for all $n\in \mathbb Z$. Given sequences $\mathcal F, \mathcal G \in \mathcal S^r ((X_n)_n)$,
define the distance
$
|\| \mathcal F - \mathcal G\||
	:= \sup_{n\in \mathbb Z} \; d_{C^r} (f_n, g_n),
$
where $d_{C^r} (f,g)$ denotes the usual $C^r$-distance between $f$ and $g$.
Given $n \ge 0$ set
$
F_n=f_{n-1}\circ \dots f_2 \circ f_1 \circ f_0
$
The \emph{(positive) orbit} of $x\in X$ is the set
$\mathcal O_{\mathcal F}^+(x)=\{F_n(x) : n\in \mathbb Z_+\}.$
In the case that $\mathcal Z=\mathbb Z$ and each element of $\mathcal F$ is invertible then the \emph{orbit} of
$x\in X$ is given by the set $\mathcal O_{\mathcal F}(x)=\{F_n(x) : n\in \mathbb Z\}$, where
$F_{-n}=f_{-n}^{-1}\circ \dots \circ f_{-2}^{-1} \circ f_{-1}^{-1}: X_0 \to X_{-n}$ for every $n\in \mathbb Z_+$.

In the present section we will state our main results.

\subsection{Statements}

\subsubsection*{Limit theorems}

An important question in ergodic theory concerns the stability of invariant measures.
In opposition to the notion of statistical stability, where one is interested in the continuity
of a specific class of invariant measures in terms of the dynamics,
here we are interested in the stability of the space of \emph{all} invariant measures.
In the case of expanding maps the existence of sequences of conjugacies
mean that, from a topological viewpoint, one can disregard the first iterates
of the dynamics. This is of particular interest in the case the sequence of dynamics
is converging.
Given a continuous observable $\phi : X \to \mathbb R$, consider the random variables $Y_n=\phi\circ F_n$
associated to the non-autonomous dynamics $\mathcal F=\{f_n\}_{n\in \mathbb Z}$. In general, the sequence
$(Y_n)_n$ may not be independent nor stationary. In particular, invariant measures for all maps in the sequence
$\cF$ seldom exist. In order to establish limit theorems for non-autonomous dynamical systems we consider
convergent subsequences: we assume the sequence $\mathcal F=\{f_n\}_{n\in \mathbb Z}$ is so that $f_n\to f$
in the $C^1$-topology.
We say that a probability measure on $X$ is \emph{$\cF$-average invariant} if
$$\lim_{n\to\infty} \frac1n\sum_{j=0}^{n-1} (f_j)_*\mu=\mu,$$
where $(f_i)_*: \mathcal M(X) \to \mathcal M(X) $ denotes
the usual push-forward map on the space of probability measures on $X$,
We note that the previous notion deals with the individual dynamics $f_j$ instead of the
concatenations $F_j$, and that a notion of invariant measures for sequential dynamics
has been defined in \cite{Kawan}.
Under the previous convergence assumption, it is not hard to check that the set of $\cF$-average invariant probability measures is a
non-empty compact set and it contains the set of $f$-invariant probability measures. This makes natural to ask wether limit theorems for the limiting dynamics $f$ propagate for $C^1$-non-autonomous dynamics. The Birkhoff irregular set of $\phi$ with respect to $\cF$ is
$$
I_{\cF,\phi}:=\Big\{  x\in X: \frac1n \sum_{j=0}^{n-1} \phi (F_j( x )) \text{\,is not convergent\,}\Big\}
$$
and, by some abuse of notation, we denote by $I_{f,\phi}$ the Birkhoff irregular set of $\phi$ with respect to $f$.
Recall that the observable $\phi$ is \emph{cohomologous to a constant $c$ (with respect to $f$)}
if there exists a continuous function $\chi: X\to \mathbb R$ such that $\phi=\chi\circ f -\chi + c$.
Our first result points in this direction by characterizing the limits of C\'esaro averages of continuous
observables and the irregular set for the non-autonomous dynamics $\cF$ in terms of the limit dynamics $f$
(we refer to Section~\ref{sec:prelim} for the necessary definitions).

\begin{theorem}\label{thm:asymptinv}
Assume that $f$ is a $C^1$-transitive Anosov diffeomorphism on a compact Riemannian manifold $X$ and let
$\mathcal F=\{f_n\}_n$ be a sequence of $C^1$ maps that is $C^1$ convergent to $f$.
Then there exists a homeomorphism $h$ such that:
\begin{enumerate}
\item if $\mu$ is $f$-invariant and $\phi \in C(X)$ then for $(h_*\mu)$-almost every $x \in X$
	 one has that $\lim_{n\to\infty} \frac1n \sum_{j=0}^{n-1} \phi (F_j( x ))$ exists
	and $\int \tilde \phi \, d\mu= \int \phi \, d\mu$;
\item if $\mu$ is $f$-invariant and ergodic then
$\lim_{n\to\infty} \frac1n \sum_{j=0}^{n-1} \delta_{F_j( x )}=\mu$
for $(h_*\mu)$-almost every $x \in X$.
\end{enumerate}
Moreover, if $\phi$ is not cohomologous to a constant with respect to $f$ then the Birkhoff irregular set
$
I_{\cF,\phi}:=\big\{  x\in X: \frac1n \sum_{j=0}^{n-1} \phi (F_j( x )) \text{\,is not convergent\,}\big\}
$
is full topological entropy, Baire generic subset of $X$.
\end{theorem}

An analogous of Theorem~\ref{thm:asymptinv} for convergent sequences of $C^1$-expanding maps
holds with the same argument used in the proof, where the existence of sequential conjugacies for $C^1$-close
Anosov diffeomorphisms is replaced by the same property for $C^1$-expanding maps.
Recent contributions on limit theorems for non-autonomous dynamics include \cite{HNTV,NTV}.
A central limit theorem holds in a neighborhood of structural stability, provided the dynamics
converge sufficiently fast. We say that a sequence of random variables satisfies the
\emph{almost sure invariance principle (ASIP)}
if there exists $\vep>0$, a sequence of random variables $(S_n)_n$
and a Brownian motion $W$ with variance $\sigma^2\ge 0$ such that $\sum_{j=0}^{n-1} \phi\circ F_j=_{\mathcal D} S_n$
and
$$
S_n= W_n + \mathcal O(n^{\frac12 -\vep})
$$
almost everywhere.
The ASIP implies in many other limit theorems as the central limit theorem CLT), the weak invariance principle (WIP) or the law of the iterated
logarithm (LIL) (see e.g. \cite{PS}).
In order to prove limit theorems we require observables to be at least H\"older continuous. We prove the following:

\begin{theorem}\label{thm:ASIP}
Assume that $f$ is a $C^1$-transitive Anosov diffeomorphism on a compact Riemannian manifold $X$, let
$\mathcal F=\{f_n\}_n$ be a sequence of $C^1$ maps and let
$a_n:=\sup_{\ell \ge n}\|f_\ell-f\|_{C_1}$. If $\phi: X \to \mathbb R$ is a $\alpha$-H\"older continuous,
$\int \phi\, d\mu=0$  and there exists $C>0$ so that $a_j \le C j^{-(\frac{1}{2}+\vep)\frac1\alpha}$ for all $j\ge 1$
then $\{ \phi \circ F_j\}$ satisfies the ASIP.
\end{theorem}

Some comments are in order. The previous result shows that the sum of the non-stationary random variables
$\{ \phi \circ F_j\}$ are strongly approximated by the Brownian motion.
We note that for a lower regularity (H\"older exponent) of the potential we require a higher
velocity of convergence given by the tail of the sequence $\cF$. We should also mention some related results.
The CLT under a stability assumption (convergence to a map) for piecewise expanding interval maps was proved in \cite{CR}.
The latter result is related to \cite[Theorem~3.1]{HNTV} where the authors obtain the ASIP for sequences of expanding maps.
In both papers, the authors used spectral methods and the analysis of compositions of the transfer operators.
Our approach differs significantly as we explore the existence of sequential conjugacies for these non-autonomous dynamics
(see Subsection~\ref{subsec:sequentialconj} for definition and more details).
The previous result is complementary to the results by Stenlund ~\cite{Stenlund} on exponential memory loss
for compositions of Anosov diffeomorphisms.

\subsubsection*{A stability result}

The notion of uniform hyperbolicity is strongly related to $C^1$-structural stability, that is, that of $C^1$-dynamics that
is topologically conjugate to all $C^1$-nearby dynamical systems (see e.g.~\cite{Hay,Ma1}).
In this subsection we shall consider the stability of non-autonomous sequences of $C^1$-Anosov diffeomorphisms.

\begin{theorem}\label{cor:Franks}
Let $X$ be a compact Riemannian manifold and $f \in \text{Diff}^r(X)$ be a $C^r$ Anosov diffeomorphism on $X$,
$r\ge 1$.  There exists $\vep>0$ so that if $\cF=\{f_n\}_{n\in \mathbb Z}$ is a sequence with
$d_{C^1}(f_n,f)<\vep$ for all $n\in\mathbb Z$ then
there exists  a sequence $(h_n)_{n\in \mathbb Z_+}$ of homeomorphisms on $X$ so that
\begin{equation}\label{eq:qconjhn}
f_{n-1} \circ \dots \circ f_1 \circ f_0\,=\,  h_n^{-1} \circ f^n \circ h
\quad \text{for every $n\in \mathbb Z_+$.}
\end{equation}
\end{theorem}

The previous theorem should be compared with the stability of $C^2$ Axiom A diffeomorphisms
with the strong transversality condition proved by Franks~\cite{Franks74} (after \cite{Robbin71}): if $f$ is
a $C^2$ Axiom A diffeomorphism with the strong transversality condition then there exists a $C^1$-open neighborhood
of $f$ and for every finite set $g_1, g_2, \dots, g_n \in \cU$ there exists a homeomorphism $h$ so that
$g_{n} \circ \dots \circ g_1\,=\,  h^{-1} \circ f^n \circ h$. The approach in \cite{Franks74} is use Banach's fixed point theorem
and to construct the conjugacy as the fixed point of a suitable contraction on a Banach space.
Some stability results for Anosov families have been announced in \cite{Acevedo2}, by a similar technique.
For that reason, the dependence of the conjugacy $h$ on increasing sequences of
diffeomorphisms $g_{1}, g_2, \circ \dots, g_n$ is not explicit.
 As we consider infinite sequences of maps we obtain a sequence of homeomorphisms $(h_n)_n$ satisfying the time-adapted
almost conjugacy condition ~\eqref{eq:qconjhn}.
A version of Theorem~\ref{cor:Franks} for sequences of $C^1$-expanding maps will appear later in Proposition~\ref{thm:C}.

\subsection{Ideas in the proofs}

In this subsection we introduce the ideas underlying the proof of the main results and detail the organization of the paper.
First we note that it is natural to expect that a \emph{quantitative} version of Theorem~\ref{cor:Franks} could be useful
to prove the ergodic stability results established in Theorems~\ref{thm:asymptinv} and ~\ref{thm:ASIP}.
Furthermore, we highlight that the existence of sequential conjugacies in Theorem~\ref{cor:Franks} seldom follows
from structural stability. Indeed, in the setting of Theorem~\ref{cor:Franks}, even though all elements of $\cF=\{f_n\}_{n\in \mathbb Z}$ are topologically conjugated to $f_0$, say for every $n$ there exists an homeomorphism  $\tilde h_n$ satisfying $f_n = \tilde h_n^{-1} \circ f_0 \circ \tilde h_{n}$, the compositions
$$
F_n = \tilde h_n^{-1} \circ f_0 \tilde \circ h_{n} \tilde  h_{n-1}^{-1} \circ \dots  \circ f_0 \circ \tilde  h_{2} \tilde  h_1^{-1} \circ f_0 \circ
\tilde h_{1} \tilde  h_0^{-1} \circ f_0  \circ \tilde h_{0} \circ f_0
$$
could behave wildly since it consists of an alternated iteration of $f_0$ with the homeomorphisms
$\tilde h_k \tilde h_{k-1}^{-1}$, which are only $C^0$-close to identity.
For any $k\in \mathcal Z$ consider the shifted sequence $\mathcal F^{(k)}=\{f_{n+k}\}_{n\in \mathcal Z}$.
Our approach to construct sequential conjugacies is to explore shadowing for sequences of hyperbolic dynamical systems
(see also Proposition~\ref{thm:C} for a similar statement in the context of sequences of expanding maps)
as follows:

\begin{theorem}\label{thm:Ainvertivel}
Let $X$ be a compact Riemannian manifold and let
$f\in \text{Diff}^{\,1}(X)$ be an Anosov diffeomorphism.
There exists a $C^1$-open neighborhood $\cU$ of $f$ so that every $\mathcal F=\{f_n\}_{n\in \mathbb Z}$
formed by elements of $\cU$ satisfies the
Lipschitz shadowing property.
Moreover:
\begin{enumerate}
\item for any $\beta>0$ there exists $\zeta>0$ so that any $\zeta$-pseudo orbit
$(x_n)_{n \in \mathbb Z}$ is $\beta$-shadowed by a unique point $x\in X$;
\item there exists $K>0$ so that if  $\vep>0$ is small so that, for any sequence $\mathcal G=\{g_n\}_n$ of elements in $\cU$ with
$|\| \mathcal F - \mathcal G \||<\vep$ there exists a unique homeomorphism
$h=h_{\mathcal F, \mathcal G}: X\to X$ so that $d_{C^0}(h_{\cF,\cG}, id) \le K |\|\cF - \cG\||$ and
\begin{equation}\label{eq:conjugacyn}
h_{ \cG^{(n)},\cF^{(n)}} \circ F_n = G_n \circ h_{\mathcal G,\mathcal F}
	\quad \text{$\forall n\in \mathbb Z_+$,}
\end{equation}
where $h_{\cF^{(n)}, \cG^{(n)}}: X \to X$ denotes the uniquely homeomorphism determined
by the sequences $\cF^{(n)}$ and $\cG^{(n)}$.
\end{enumerate}
In particular $d_{C^0}(h_{\cF,\cG}, id) \to 0$ as $|\| \cF - \mathcal G\|| \to 0$.
\end{theorem}

We will use this \emph{quantitative} version of
Theorem~\ref{cor:Franks} to prove the ergodic stability for sequences of hyperbolic maps.
This paper is organized as follows. In Section~\ref{sec:prelim} we recall some preliminary notions of stability, shadowing
and entropy for sequential dynamical systems. Section~\ref{sec:shadowingnaut} we prove some shadowing results for
both sequences of $C^1$-expanding maps and sequences of nearby $C^1$-Anosov diffeomorphisms. This allow us to
construct sequential conjugacies, which we explore in Section~\ref{sec:cors} to prove the main results on the ergodic
stability for convergent sequences of hyperbolic dynamics.

\section{Preliminaries}\label{sec:prelim}

\subsection{Sequential and almost conjugacies}\label{subsec:sequentialconj}

A sequence $\cF$ of continuous maps acting on a compact metric space $X$ is \emph{topologically stable} if for
every $\vep>0$ there exists $\delta>0$ such that for every sequence $\mathcal G=\{g_n\}_n$ so that $|\| \mathcal F - \mathcal G\||<\delta$ there exists a continuous map $h : X \to X$ so that $\|h-id\|_{C^0}<\vep$ and $d(F_n(h(x)), G_n(x))<\vep$ for all
$n\in \mathbb Z$.
Moreover, $\mathcal F=\{f_n\}_{n\in \mathbb Z_+}$ is an \emph{expansive} sequence of maps if there exists $\vep>0$
so that for any distinct points $x,y\in X_0$ there exists $n\in \mathbb Z_+$ such that $d(F_n(x),F_n(y))>\vep$.
It is known that any positively expansive sequential dynamics admits an adapted metric on which it actually expands distances \cite{Kawan}
and that positively expansive non-autonomous  dynamical systems acting on a compact metric space with the shadowing property are
topologically stable \cite{DR14}.

Given $\beta> 0$ and the sequences of continuous maps $\mathcal F= \{f_n\}_{n \geq 1}$ and $\mathcal G= \{g_n\}_{n \geq 1}$
on a complete metric space $(X,d)$, we say that an homeomorphism $h:X \to X$ is a \emph{$\beta$-quasi-conjugacy between $\mathcal F$ and
$\mathcal G$} if
$$
d_{C^0} (h\circ F_n , G_n \circ h) \leq \beta
$$
for every $n\in \mathbb Z_+$, where $d_{C^0}(f, g)=\|f-g\|_{C^0}$ denotes the distance in the $C^0$-topology.
The second notion does not require compactness nor the maps to act on the same compact metric space.
Given
sequences $\mathcal F=\{f_n\}_{n\in \mathbb Z}$ and $\mathcal G=\{g_n\}_{n\in \mathbb Z}$  of continuous maps acting on complete metric spaces $(X_n)_n$, we say that a sequence $\mathcal H=\{h_n\}_n$ of homeomorphisms $h_n : X_n \to X_n$ is a
\emph{sequential conjugacy between $\cF$ and $\cG$} provided that the maps $F_n : X_0 \to X_n$ and $G_n : X_0 \to X_n$
to satisfy
$$
h_n \circ F_n = G_n \circ h_0 \;\text{for all $n\in \mathbb Z$}.
$$
Each of the maps $h_n$ in the notion of sequential conjugacies are defined in terms of the infinite sequence of maps
$\cF^{(n)}$.

Similarly to the classical setting, in our setting sequential conjugacies $C^0$-close to the identity are unique
(recall Theorem~\ref{thm:Ainvertivel}). Moreover, if $\cF=\{f\}_n$ and $\cG=\{g_n\}_n$ are sequences
of dynamical systems on a compact metric space $X$ that admit a unique sequential conjugacy $C^0$-close
to the identity and the sequential conjugacies are constant (i.e. $h_n=h: X\to X$ for all $n$) then the sequences $\cF$ and $\cG$
are constant.  Indeed, if this is the case, $h$ is a conjugacy between $f$ and $g_1$ (hence between $f^2$ and $g_1^2$)
and between $f^2$ and $g_2\circ g_1$. By uniqueness of the conjugacies $C^0$-close to identity we conclude that
$g_2=g_1$. Applying this argument recursively we conclude that $\cG=\{g_1\}_n$ is constant.
This fact also shows that some flexibility in the definition of conjugacies for sequential dynamics would be necessary.
The flexibility of the concept of sequential conjugacies, for dynamics acting on different metric spaces,
allowed us to describe a leafwise shadowing property for invariant foliations of partially hyperbolic dynamics \cite{CRV1}.

\subsection{Shadowing }\label{subsec:shadow}
Our first main results concern the stability of Anosov  sequences.
Fix $r\ge 1$, let $(X_n)_n$ be a sequence of compact Riemannian manifolds
and let $\mathcal S^r ((X_n)_n)$ denote the space of sequences $\{f_n\}_n$
of $C^r$-differentiable maps $f_n: X_n \to X_{n+1}$. We say that a sequence
 $\mathcal F=\{f_n\}_{n\in \mathbb Z}$ is an \emph{Anosov sequence} if
there exists $a>0$, for every $n\in \mathbb Z$ there exists a continuous decomposition of the
tangent bundle $T X_n=E_n^+ \oplus E_n^-$ (of constant dimension),
there exist cone fields
$$
\mathcal C^+_{a,n}(x) =\big\{ v = v^+ + v^- \in E^+_n(x) \oplus E^-_n(x) \colon \| v^- \| \le a \| v^+\| \big\}
$$
and
$$
\mathcal C^-_{a,n}(x) =\big\{ v = v^+ + v^- \in E^+_n(x) \oplus E^-_n(x) \colon \| v^+ \| \le a \| v^-\| \big\}
$$
and constants $\lambda_n \in (0,1)$ so that:
\begin{itemize}
\item[a)]
	$Df_n (x) \, \mathcal C^+_{a,n}(x) \subset \mathcal C^+_{\lambda_n a, \, n+1}(f_n(x))$ and
	$Df_n (x)^{-1} \, \mathcal C^-_{a,n+1}(f_n(x)) \subset \mathcal C^-_{\lambda_n a,\,  n}(x)$
\item[b)]
	$\|Df_n^{-1}(x) v \| \ge \lambda_n^{-1} \|v\|$ for every $v\in \mathcal C^-_{a,n+1}(f_n(x))$ and  \\
	 $\|Df_n(x) v \| \ge \lambda_n^{-1} \|v\|$ for every $v\in \mathcal C^+_{a,n}(x) $
\end{itemize}
for every $x\in X_n$ and $n\in \mathbb Z_+$. We refer to $a>0$ as the \emph{diameter} of the cone fields.
It is clear that if  $\mathcal F$ is an Anosov sequence then  $\mathcal F^{(k)}$ is also an Anosov
sequence for every $k\in\mathbb Z$.
Moreover, a constant sequence $\{f\}_{n\in \mathbb Z}$ is an Anosov sequence if and only if the diffeomorphism $f$
is Anosov. Here we will be interested in Anosov sequences formed by diffeomorphisms in a $C^1$-neighborhood of
some fixed Anosov diffeomorphism.

As uniform hyperbolicity can be characterized by the existence of stable and unstable cone fields,
if $f$ is an Anosov diffeomorphism there exists a $C^1$-open neighborhood $\cU$ of $f$ such that
every sequence $\cF=\{f_n\}$ formed by elements of $\cU$ is an Anosov sequence.
We refer the reader to \cite{KH} for the $C^1$-robustness and stability of Anosov diffeomorphisms, and to
Subsection~\ref{sec:geometric} for some of the geometrical properties of Anosov sequences.

In order to state our  main results on shadowing and stability of non-autonomous dynamical systems we recall
some necessary notions.
Given $\delta>0$, we say that $(x_n)_{n=0}^k$ is a
\emph{$\delta$-pseudo orbit} for $\mathcal F=\{f_n\}_{n\in \mathbb Z_+}$
if $x_n \in X_n$ and $d( f_n(x_n), x_{n+1} )<\delta$ for every $0\leq n\leq k-1$.
We say that the sequence of maps $\mathcal F=\{f_n\}_{n\in \mathbb Z_+}$ has the \emph{shadowing property}
if for every $\varepsilon>0$ there exists $\delta>0$ such that for any $\delta$-pseudo-orbit
$(x_n)_{n=0}^k$ there exists $x\in X$ so that its $\mathcal F$-orbit $\varepsilon$-shadows the
sequence $(x_n)_{n=0}^k$, that is,
$
d(  F_n (x), x_{n+1})<\varepsilon
\,\text{ for every } 0\leq n \leq k-1.
$
Moreover, we say that $\mathcal F=\{f_n\}_{n\in \mathbb Z_+}$ has the \emph{Lipschitz shadowing property} if there exists a uniform
constant $L>0$ so that one can choose $\delta=L \vep$ above.
Finally, if $X_n=X$ for every $n$, we say that the sequence $\mathcal F=\{f_n\}_{n\in \mathbb Z_+}$ has the
\emph{periodic shadowing property} if for any $\vep>0$ there exists $\delta>0$ so that any
$\delta$-pseudo orbit $(x_n)_{n=0}^k$ satisfying $x_0=x_k$ is $\vep$-shadowed by a fixed point $x\in X$ for
$F_k=f_k \circ \dots \circ f_2\circ f_1\circ f_0$. The previous notions are often referred as
finite shadowing properties since consider finite pseudo-orbits. Nevertheless, in locally compact spaces it is
a well known fact that the finite shadowing orbit property is equivalent to the shadowing  property using infinite pseudo-orbits.

\subsection{Topological entropy }
In their seminal work, Kolyada and Snoha~\cite{K-S} introduced and studied a concept of entropy for
non-autonomous dynamical systems and prove, among other results, that the entropy is concentrated
in the non-wandering set. The non-wandering set for non autonomous dynamical systems is a compact set
but, in general, it misses to be invariant by the sequence of dynamical systems. This makes the problems
of proving stability and finding conjugacies for nearby dynamics a hard topic. Let us recall some necessary results
from \cite{K-S}.
Let $\cF=\{f_n\}_n$ be a sequence of maps on a compact metric space $(X,d)$. For every $n\ge 1$ consider the
distance $d_n(x,y):=\max_{0\le j \le n-1} d( F_j(x), F_j(y))$. A set $E\subset X$ is $(n,\vep)$-\emph{separated}
for $\cF$ if $d_n(x,y)>\vep$ for every $x,y\in E$ with $x\neq y$. For $Z\subset X$ define
$$
s_n(\cF,\vep, Z)= \max\{\# E \colon E \; \text{is a}\; (n,\vep) \, \text{separated set in} \, Z\}
$$
and the \emph{topological entropy of $\cF$ on $Z\subset X$} by
\begin{equation}\label{def:entropy}
h_{Z}(\cF)
	= \lim_{\vep\to 0}\limsup_{n\to\infty} \frac1n \log s_n(\cF,\vep, Z),
\end{equation}
and set $h_{top}(\cF)=h_{X}(\cF)$.

The topological entropy of sequences $\cF=\{f_n\}_n$ and $\cG=\{g_n\}_n$ is not determined by the individual dynamics.
Indeed, there are examples where each $f_n$ and $g_n$ are topologically conjugate for all $n$ but $h_{top}(\cF) \neq h_{top}(\cG)$
(cf. \cite[Section 5]{K-S}).

A pair of sequences $\cF=\{f_n\}_n$ and $\cG=\{g_n\}_n$ is \emph{equiconjugate} if there
exists a sequence of homeomorphism $(\tilde h_n)_n$ such that:\,  (i) $\tilde h_{n+1}\circ f_n= g_n\circ \tilde h_n$, and
(ii) the sequences $(\tilde h_n)_n$ and $(\tilde h_n^{-1})_n$ are equicontinuous. It is not hard to check that if $f$ is structurally stable
and $(f_n)_n$ are $C^1$-close and convergent to $f$ then $\cF=\{f\}_n$ and $\cG=\{f_n\}_n$ are equiconjugate.
Moreover, the following holds:

\begin{proposition}\cite[Theorem~B]{K-S}\label{propKS}
Let $\cF=\{f_n\}_n$ and $\cG=\{g_n\}_n$ be sequences of continuous maps on a compact metric space $X$ and $Y$, respectively.
If $\cF$ and $\cG$ are equiconjugate then $h_{top}(\cF) =h_{top}(\cG)$.
\end{proposition}

\section{Shadowing for non-autonomous dynamics}\label{sec:shadowingnaut}

\subsection{Stability of non-autonomous expanding maps on compact metric spaces}

Let $(X_n,d_n)$ be a sequence of complete metric spaces and
let $\mathcal F=\{f_n\}_{n\in \mathbb Z_+}$ be a sequence of continuous and onto maps
$f_n: X_n \to X_{n+1}$.
 We say that $\mathcal F$ is a \emph{sequence of expanding maps}
if there are $\delta_0>0$
and a sequence $(\lambda_n)_{n\in \mathbb Z_+}$ of constants in $(0,1)$ so that the following holds:
for any $n\in \mathbb Z_+$, $x\in X_{n+1}$ and $x_i\in f_n^{-1}(x)$ there exists a well defined inverse branch
$f_{n,x_i}^{-1}: B(x,\delta_0) \to V_{x_{i}}$ (open neighborhood of $x_{i}$) so that
$
d (  f_{n,x_i}^{-1}(y),  f_{n,x_i}^{-1}(z) )
	\le \lambda_n \, d(y,z)
$
for every $y,z\in B(x,\delta_0)$
Here, for notational simplicity, we omit the metrics $d_n$ representing them by $d$.
In what follows we observe that any sequence $\mathcal F$ of expansive maps on compact metric spaces admit
a uniform lower bound on the separation time.
More precisely:

\begin{lemma}\label{lemma:expansiveness2}
Let $\mathcal F=\{f_n\}_{n\in\mathbb Z_+}$ be an expansive sequence of continuous maps
$f_n: X_n \to X_{n+1}$ acting  on metric spaces and let $\vep_0$ be an expansiveness constant for $\mathcal F$.
If $X_0$ is compact then for any $\delta>0$ there exists $N\in \mathbb Z_+$ such that, if $x,y\in X_0$ satisfy
$d(F_n(x),F_n(y))<\vep_0$ for every  $0\leq n\leq N$ then  $d(x,y)<\delta$.
\end{lemma}

\begin{proof}
We prove the lemma by contradiction.  Assume there exists $\delta>0$ and, for every $j\ge 0$, there are
$x_j,y_j \in X_0$ with $d(x_j,y_j) >\delta$ and  $d(F_n(x_j),F_n(y_j))<\vep_0$ for every  $0\leq n\leq j$.
Since $X_0$ is compact we may assume (up to consider subsequences) that $x_j \to x \in X_0$
and $y_j \to y \in X_0$.
By continuity, taking $j \to\infty$ we get that $x\ne y$ and
$
d(F_n(x),F_n(y))\le \vep_0 \; \text{for every  $n\in \mathbb Z_+$},
$
which contradicts the expansiveness property. This proves the lemma.
\end{proof}

We can now state our second result on the stability of sequences of expanding maps.

\begin{proposition}[Existence of sequential conjugacies] \label{thm:C}
Let $\mathcal F=\{f_n\}_{n\in \mathbb Z_+}$ be a sequence of $C^1$-expanding maps acting on compact Riemannian manifolds $X_n$ with
contraction rates $(\lambda_n)_{n\in \mathbb Z_+}$ for inverse branches satisfying $\sup_{n\in \mathbb Z_+} \lambda_n <1$.
There exists $\vep>0$ so that, for any sequence $\mathcal G$ of $C^1$ maps satisfying
$|\| \mathcal F - \mathcal G \||<\vep$ there exists a unique homeomorphism
$h=h_{\mathcal F, \mathcal G}: X_0 \to X_0$ so that
\begin{equation}\label{eq:conjugacyn}
h_{ \cG^{(n)},\cF^{(n)}} \circ F_n = G_n \circ h_{\mathcal G,\mathcal F}
	\quad \text{$\forall n\in \mathbb Z_+$,}
\end{equation}
where $h_{\cF^{(n)}, \cG^{(n)}}: X_n \to X_n$ denotes the uniquely homeomorphism determined
by the sequences $\cF^{(n)}$ and $\cG^{(n)}$.
Moreover,  there exists $L>0$ so that $\|h_{\cF, \cG}-id\|_{C^0} \le L |\| \cG - \cF\||$.
\end{proposition}

\begin{proof}
Let $\mathcal F=\{f_n\}_{n\in \mathbb Z_+}$ be as above
and let $\vep_1>0$ be small such that any sequence
of $C^1$ maps $\mathcal G$ satisfying $|\| \mathcal F - \mathcal G \||<\vep_1$ is a sequence of $C^1$-expanding maps
(such a constant exists since the set of expanding maps is $C^1$-open and  $\lambda:=\sup_{n\in \mathbb Z_+} \lambda_n <1$). Reduce
$\vep_1>0$, if necessary, so that every sequence $\cG$
as above is expansive with uniform expansiveness constant $\vep_0>0$.

Fix $0<\vep < \frac14 \min\{\vep_0, \vep_1\}$ and let $L\ge 1$ be given by the Lipschitz shadowing property (cf. Proposition~\ref{prop:A}).
If $\cG$ is any sequence of $C^1$ expanding maps $\mathcal G$ satisfying $|\| \mathcal F - \mathcal G \||<\vep/L$
and $x\in X_0$ then the sequence $(G_n(x))_{n\in \mathbb Z_+}$ forms a $\vep/L$-pseudo orbit with respect to the sequence $\mathcal F$, as
$$
d( f_n ( G_n(x) ) , G_{n+1}(x) )
	= d( f_n ( G_n(x) ) , g_n (G_n(x)) ) \le \|f_n -g_n \|_{C^0} < \frac{\vep}{L}
$$
for every $n\in \mathbb Z_+$.
Hence, there exists a unique point $h_{\mathcal F,\mathcal G}(x) \in X_0$ so that
\begin{equation}\label{conjug}
d (F_n (h_{\mathcal F,\mathcal G}(x)),  \, G_n(x) ) <\vep
	\quad\text{ for every } n\in \mathbb Z_+
\end{equation}
(see Figure~\ref{fig2} below).
Reversing the role of $\cF$ and $\mathcal G$ and replacing $x$ by $h_{\mathcal F,\mathcal G}(x)$,
we deduce
that there exists a unique point $h_{\mathcal G,\mathcal F}( h_{\mathcal F,\mathcal G}(x) ) \in X_0$ so that
$$
d (G_n (h_{\mathcal G,\mathcal F}(h_{\mathcal F,\mathcal G}(x))),  \, F_n(h_{\mathcal F,\mathcal G}(x)) ) <\vep
	\quad\text{ for every } n\in \mathbb Z_+.
$$
\begin{figure}[htb]\label{fig2}
\begin{center}
  \includegraphics[width=11cm,height=4.3cm]{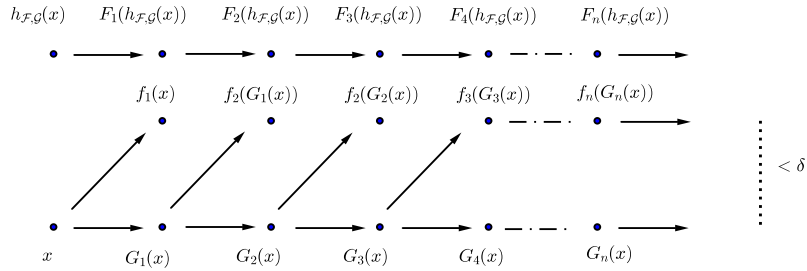}
\caption{$(G_n(x))_{n\ge 0}$ as $\delta$-pseudo-orbit with respect to $\mathcal F=\{f_n\}_{n\in \mathbb Z_+}$
	and shadowing point $h_{\mathcal F,\mathcal G}(x)\in X_0$}
\label{figure}
\end{center}
\end{figure}
As $|\| \mathcal F - \mathcal G\||<\vep$, by triangular inequality we get
$
d (G_n (h_{\mathcal G,\mathcal F}(h_{\mathcal F,\mathcal G}(x))), G_n(x))<2\vep <\vep_0
$
for every $n\in \mathbb Z_+$.
Since $\vep_0$ is an expansiveness constant for $\cG$, the latter assures that
 $h_{\mathcal G,\mathcal F}(h_{\mathcal F,\mathcal G}(x))=x$, proving that $h_{\mathcal F,\mathcal G}$ is invertible and
 $h_{\mathcal F,\mathcal G}^{-1}=h_{\mathcal G,\mathcal F}$.
\begin{figure}[htb]\label{fig3}
\begin{center}
  \includegraphics[width=11cm,height=4.3cm]{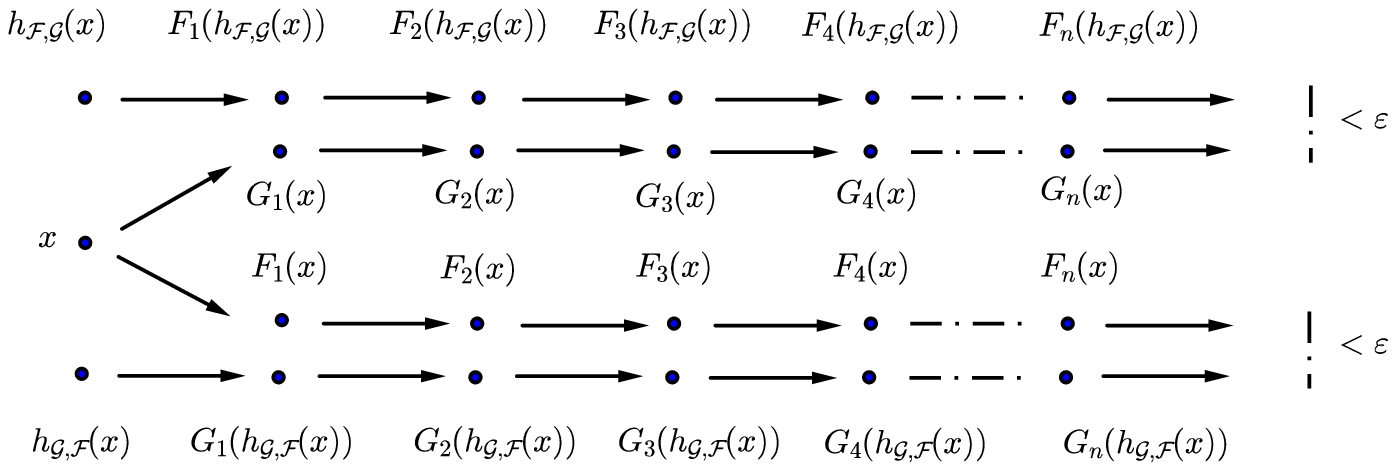}
\caption{Shadowing points $h_{\cF,\cG}(x)$ and $h_{\cG,\cF}(x)$ on $X_0$}
\label{figure}
\end{center}
\end{figure}
 Moreover, taking $n=0$ in \eqref{conjug} we get that
$\|h_{\cF,\cG}-id_{X_0}\|_{C^0} \leq \vep$.
Now we prove that $h=h_{\mathcal F, \mathcal G}: X_0 \to X_0$ is an homeomorphism.
Take $0<\delta<\vep_0/4$ and let $N=N(\delta)\ge 1$ given by Lemma~\ref{lemma:expansiveness2}.
As the spaces $X_n$ are compact, the set of functions $\{ G_1, \dots, G_N\}$ is equicontinuous:
there exists $\eta>0$ so that if $d(x,y)<\eta$ then $d(G_n(x),G_n(y))<\delta$ for every $0\le n \le N$.
Thus, if $d(x,y)<\eta$ and $0\le n \le N$ then
\begin{align*}
d(F_n(h_{\mathcal F,\mathcal G}(x)),F_n(h_{\mathcal F,\mathcal G}(y)))
          & \le d(F_n(h_{\mathcal F,\mathcal G}(x)),G_n(x))+d(F_n(h_{\mathcal F,\mathcal G}(y)),G_n(y))\\
          &+d(G_n(x),G_n(y))
          \leq 2 \vep +\delta
          < \vep_0.
\end{align*}
Lemma~\ref{lemma:expansiveness2} implies $d(h_{\mathcal F, \mathcal G}(x),h_{\mathcal F, \mathcal G}(y))<\delta$
and the continuity of  $h_{\mathcal F, \mathcal G}$ follows. By a similar argument, or using the fact that $X_0$ is a compact
metric space and  $h_{\mathcal F, \mathcal G}$ is a continuous bijection, we conclude the continuity of its inverse
$h_{\mathcal G, \mathcal F}$.

Finally, we are left to prove the conjugacy relation ~\eqref{eq:conjugacyn}.
Clearly $|\| \cF^{(n)} - \cG^{(n)} \|| <\vep/L$ for every $n\in \mathbb Z_+$.
Recalling that
$
F^{(n)}_k = f_{n+k-1} \dots f_{n}
$ and
$	G^{(n)}_k = g_{n+k-1} \dots g_{n}$
for every $k\ge 0$,
we note that similar computations as before yield that the orbit of the point
$F_n(h_{\mathcal F, \mathcal G}(x) \in X_n$ by the sequence $\cF^{(n)}$
is an $\vep/L$-pseudo-orbit with respect to the sequence $\cG^{(n)}$. In particular, there exists a unique point
$h_{\cG^{(n)}, \cF^{(n)}}(\,  F_n(h_{\mathcal F, \mathcal G}(x))  \, ) \in X_n$ for which
\begin{equation}\label{eq:shadowingshifted}
d(  G_k^{(n)} (\,  h_{\cG^{(n)}, \cF^{(n)}}( F_n(h_{\mathcal F, \mathcal G}(x))  \, ) ,
	  F_k^{(n)}( F_n(h_{\mathcal F, \mathcal G}(x)) )  < \vep
\end{equation}
for every $k\ge 0$. Combining equations \eqref{conjug} and ~\eqref{eq:shadowingshifted} (recall that
 $F_k^{(n)}( F_n) = F_{n+k}$ and $G_k^{(n)}( G_n) = G_{n+k}$ for every $k\ge 0$),
by triangular inequality
$$
d( G^{(n)}_k (\,  h_{\cG^{(n)}, \cF^{(n)}}( F_n(h_{\mathcal F, \mathcal G}(x))  \, )) ,
	 G^{(n)}_k (G_n(x)) )  < 2\vep <\vep_0
$$
for every $k\ge 0$. Using once more that $\vep_0$ is an expansiveness constant for the sequence $\cG^{(n)}$
we deduce that
$
h_{\cG^{(n)}, \cF^{(n)}} \circ F_n \circ h_{\mathcal F, \mathcal G}(x) = G_n(x)$ for every $n\in \mathbb Z_+$ and
every $x\in X_0$. In other words,
$
h_{\cG^{(n)}, \cF^{(n)}} \circ F_n (x) = G_n\circ h_{\mathcal G, \mathcal F}$
for all $n\in \mathbb Z_+$,
which finishes the proof of the propostion.
\end{proof}

\begin{remark}\label{rmk:conjugacy}
Proposition~\ref{thm:C} means that the conjugacies $(h_{\cF^{(n)}, \cG^{(n)}})_{n\in \mathbb Z_+}$
improve (meaning that $h_{\mathcal F, \mathcal G}$ is $C^0$-convergent to the identity map) in the case
that $|\| \cF^{(n)} - \cG^{(n)}\|| \to 0$ as $n$ tends
to infinity.
If $f,g$ are $C^1$-expanding maps on a compact Riemannian manifold $X$,
$\cF=\{f\}_{n\in \mathbb Z_+}$ and $\cG=\{g\}_{n\in \mathbb Z_+}$ then
$F_n=f^n$ and  $G_n=g^n$ for every $n\in \mathbb Z_+$ and the sequence of conjugacies
$(h_{\cF^{(n)}, \cG^{(n)}} )_{n\in \mathbb Z_+}$ is constant to the conjugacy $h_{\mathcal F, \mathcal G}$ between
$f$ and $g$.
More generally, if there exists $N\ge 1$ so that $\cF^{(N)}=\cF$ (e.g. $\cF=\{f,g,f,g, f,g,\dots\}$ and $N=2$)
then the conjugacies $(h_n)_{n\in \mathbb Z_+}$ are $N$-periodic: $h_{n+N}=h_n$ for every $n\in \mathbb Z_+$.
\end{remark}

In the special case of periodic sequences of expanding maps we derive the following:

\begin{maincorollary}\label{cor:periodic}
Assume that $f$ is a $C^1$-expanding map on a compact Riemannian manifold $X$ and let
$\mathcal F=\{f_n\}_n$ be a $N$-periodic sequence of expanding maps such that
$\cF$ and $\cG=\{f\}_{n\in\mathbb Z}$ are sequentially conjugate. Then there are homeomorphisms $(h_i)_{i=0\dots N-1}$
so that
$$
\lim_{n\to\infty} \frac1n \sum_{j=0}^{n-1} \delta_{F_j(x)}
	 = \frac1N\sum_{i=0}^{N-1} (h_i)_*\mu
	 \quad \text{for $\mu$-a.e. $x\in X$.}
$$
\end{maincorollary}

\begin{proof}
Assume that $\cF=\{f_n\}_{n\in \mathbb Z}$ is a $N$-periodic sequence (that is, $f_{n+N}=f_n$ for every $n\in\mathbb Z$),
and that $\cF$ and $\cG=\{f\}_{n\in\mathbb Z}$ are sequentially conjugate: for every $n$ there exists a homeomorphism $h_n$
so that
$
F_n=h_n \circ f^n \circ h_0^{-1}.
$
 By Remark~\ref{rmk:conjugacy} the sequence
of conjugacies $(h_n)_n$ is also $N$-periodic.
Let $\mu$ be a $f$-invariant probability (hence $f^N$-invariant) that is ergodic with respect to $f^N$. If $\phi\in C^0(X)$,
$\tilde \mu:=(h_0)_*\mu$ and $h_0(y)=x$ then
\begin{align}
\sum_{j=0}^{n-1} \phi(F_j(x))
	& = \sum_{j=0}^{n-1} \phi \circ h_j \circ f^j (h_0^{-1}(x)) \nonumber \\
	& = \sum_{i=0}^{N-1} \sum_{\ell=0}^{\big[\frac{n}{N}\big]} \sum_{0\le \ell N+i \le n-1}(\phi \circ h_i) ( f^{\ell N+i}(y))
	\nonumber\\
	& = \sum_{i=0}^{N-1} \sum_{\ell=0}^{\big[\frac{n}{N}\big]} \sum_{0\le \ell N+i \le n-1} \psi_i( f^{\ell N}(y))
	\label{eq:estimate-periodic}
\end{align}
for $\tilde\mu$-almost every $y$, where $\psi_i=\phi \circ h_i \circ f^i$ for every $0\le i \le N-1$. Since $\mu$ is
ergodic for $f^N$ then there exists a $\tilde \mu$-full measure subset of points $y$ for which
$
\lim_{k\to\infty } \frac1k\sum_{j=0}^{k-1} \varphi (f^{jN}(y)) =\int \varphi \, d\mu
$
for every continuous $\varphi: X \to\mathbb R$.
Together with ~\eqref{eq:estimate-periodic} and the $f$-invariance of $\mu$,
this proves that
\begin{align}
\lim_{n\to\infty} \frac1n \sum_{j=0}^{n-1} \phi(F_j(x))
	 = \int \phi \; d \Big( \frac1N\sum_{i=0}^{N-1} (h_i)_*\mu\Big)
\end{align}
for $\mu$-almost every $x\in X$ and every $\phi \in C^0(X)$. In other words,
$$
\lim_{n\to\infty} \frac1n \sum_{j=0}^{n-1} \delta_{F_j(x)}
	 = \frac1N\sum_{i=0}^{N-1} (h_i)_*\mu
$$
for $\mu$-almost every $x\in X$.
\end{proof}

The following result asserts that orbits of $\beta$-quasi-conjugate $\mathcal F$ and
$\mathcal G$ (up to the $\beta$-quasi-conjugacy) remain within always within distance $\beta$ from each other
or, equivalently, the orbits are indistinguishable at scale $\beta$. More precisely:

\begin{proposition}[Existence of quasi-conjugacies]\label{thm:quasi-conexp}
Let $\mathcal F=\{f_n\}_{n\in\mathbb Z_+}$ be a sequence of $C^1$-expanding maps acting on a compact Riemannian manifold $X$
with contraction rates of inverse branches $(\lambda_n)_{n\in \mathbb Z_+}$ satisfying $\lambda:= \sup_{n\in \mathbb Z_+} \lambda_n <1$.
Then, for all sufficiently small  $\vep>0$ and  for any sequence of $C^1$ maps $\mathcal G$ satisfying
$|\| \mathcal F - \mathcal G \||<\vep$ there exists a $\frac{2\lambda}{1- \lambda} \vep$-quasi-conjugacy
$h: X \to X$  between $\mathcal F$ and $\mathcal G$.
Moreover, $\|h- id\|_{C^0} \to 0$ as $\mathcal G$ tends to $\mathcal F$.
\end{proposition}

\begin{proof}
Let $\vep>0$ be small enough so that any sequence $\mathcal G$ of $C^1$ maps
satisfying $|\| \mathcal F - \mathcal G \||<\vep$ is a sequence of $C^1$-expanding maps.
Let $\delta_0>0$ be a uniform lower bound for the radius of the inverse branches domain for the
expanding maps all such sequences $\mathcal G$.

Take $0<\delta<\delta_0/2$ and take $\vep := (1- \lambda) \delta/\lambda>0$.
Suppose $|\|\mathcal F- \mathcal G\||< \vep$. Given any $x_0=x \in X$ set $x_n:= F_n(x)$ for every $n\in \mathbb Z_+$.
As before, for every $n\in \mathbb Z_+$ let $f_{n,x_n}^{-1}$ denote the inverse branch of $f_n$ such that
$f_{n,x_{n}}^{-1}(x_{n+1})= x_{n}$.
For notational simplicity, let $g_n^{-1}$ denote the inverse branch of $g_n$, whose domain contains
$B(x_{n+1},\delta)$, which is $C^0$-closer to $f_{n,x_n}^{-1}$.
We claim that $g_n^{-1}(B(x_{n+1}, \delta)) \subset B(x_{n}, \delta)$ for all $n\in \mathbb Z_+$.
In fact, using
$$
d(g_{n}(x_{n}),  x_{n+1}) = d(g_{n}(x_{n}),  f_n(x_{n})) < \vep = \frac{1- \lambda}{\lambda}\delta
$$
we conclude that, for any $z \in B(g_{n}(x_{n}), \vep+ \delta)$ (in particular for points of $B(x_{n+1}, \delta)$),
$$
d(g_n^{-1}(z), x_{n} ) \le \lambda  (\vep+ \delta) \leq (1- \lambda)\delta + \lambda \delta=
\delta.
$$
Hence, if we define $\hat G_n:= g_1^{-1} \circ \dots \circ g_n^{-1}$ then the sets
$Y_n:=  \hat G_n(\overline{B(x_{n+1}, \delta))}$ form a nested sequence of
closed sets with diameter smaller or equal to  $2 \lambda^n \delta$. Thus, there is a
unique point $h(x) \in X$ such that $d(G_n(h(x)), F_n(x)) = d(G_n( h(x)), x_n) \leq \delta$
for every $n \in \mathbb{Z}_+$.
As $x\in X$ was chosen arbitrary, the map $h: X \to X$ defined by the previous construction satisfies
$\|h-id\|_{C^0}<\delta$.
Moreover, by triangular inequality,
$$
d(G_n(h(x)), h (F_n(x)))
	\leq d(G_n( h( x)), F_n(x))+ d(F_n(x), h \circ F_n(x)) \leq 2\delta
	=  \frac{2\lambda}{1- \lambda} \vep.
$$
We proceed to prove the continuity of $h$. Given  $\tilde \vep> 0$  take $N> 0$ such that  $2\la^{N} \delta< \tilde \vep$.
Now, by uniform continuity of the maps $F_j$ with $j\le N$ there exists $\tilde \delta> 0$ such that
if $d(x, y)< \tilde \delta$ then
$d(F_j(x), F_j(y))< \delta \; \text{for every $j= 0, \dots, N$.}$
 So,
$$
\hat G_n(\overline{B( F_{n+1}(x), 2\delta)}) \supset \hat G_n(\overline{B(F_{n+1} (y), \delta)}) \cup \hat G_n(\overline{B( F_{n+1}(x), \delta)})
$$
and contains both points $h(x)$ and $h( y)$. Since $\diam (\hat G_n(\overline{B(F_{n+1} (x)), 2 \delta)})< \tilde \vep$
this shows that $d(h(x), h(y))< \tilde \vep$, and implies on the continuity of $h$.

As we proved that there exists exactly one point $h(x)$ whose $\mathcal G$-orbit that $\delta$-shadows the
$\mathcal F$-orbit of a point $x \in X$, exchanging the roles of $\mathcal F$ and $\mathcal G$, one can use
the same argument as in the proof of Proposition~\ref{thm:C} to
assure that there exists a unique point $h^{-1}(h(x))=x$ that $\delta$-shadows the
$\mathcal G$-orbit of the point $h(x) \in X$ and, consequently, to deduce that $h$ is a homeomorphism.
It is immediate that $h \to id$ as $\mathcal G \to \mathcal F$. This finishes the proof of the proposition.
\end{proof}

We observe that the stability notions in the statement of Propositions~\ref{thm:C} and~\ref{thm:quasi-conexp} are unrelated, thus these cannot be obtained one from each other.  One of the advantages of Proposition~\ref{thm:C} is to observe
that time-dependent conjugacies become smaller as time evolves for sequences that are asymptotic. An advantage
of Proposition~\ref{thm:quasi-conexp} is to obtain quasi-conjugacies and to compute the proximity of the quasi-conjugacy
from the identity in terms of contracting rates for the sequence, which is the best one can hope computationally.
Although quasi-conjugacies need not unique, this is the case for stably expansive sequences $\cF$:
if $\vep$ is an expansiveness constant for all sequences $\mathcal G$ arbitrarily close to $\cF$, $h$ is a $\beta$-quasi-conjugacy
between $\cF$ and $\cG$ with $0<\beta<\vep/4$
and $\tilde h:X \to X$ is a homeomorphism satisfying $d_{C^0}(h,\tilde h) \ll \vep/4$ then
$$
d_{C^0} (\tilde h\circ F_n , G_n \circ \tilde h)
	\geq
	d_{C^0} (G_n \circ h , G_n \circ \tilde h)
	- d_{C^0}(h, \tilde h)
	-
	\beta
	\geq
	d_{C^0} (G_n \circ h , G_n \circ \tilde h) - \vep/2
$$
which is larger $\vep /3$ for some $n\in \mathbb Z_+$ (by the $\vep$-expansiveness of the sequence $\cG$).

\subsection{Shadowing and stability for Anosov sequences}

In this section we prove shadowing and stability for sequences of Anosov diffeomorphisms
as stated in Theorem~\ref{thm:Ainvertivel}.

\subsubsection{Invariant manifolds}\label{sec:geometric}

Assume that $\cF=\{f_{n}\}_{n\in \mathbb Z}$ is an Anosov sequence (recall Subsection~\ref{subsec:shadow}
for the definition). Recall the notations
$\cF^{(k)}=\{f_{k+n}\}_{n\in \mathbb Z}$, $F_{-n}^{(k)}=f_{k-n}^{-1}\circ \dots \circ f_{k-2}^{-1} \circ f_{k-1}^{-1}$
and $F_n^{(k)}=f_{n+k-1} \circ \dots \circ f_{k+1}\circ f_k$, for all $k,n\in \mathbb Z$. The existence of invariant
cone fields with uniform expansion and contraction imply on the following properties:
\begin{enumerate}
\item for every $k\in \mathbb Z$ and $x\in X_k$, the subspaces
	$$
	E_{\cF^{(k)}}^u(x) := \bigcap_{n \ge 0} D F^{(k-n)}_{n}  (F^{(k)}_{-n}(x)) \; \mathcal C^u_{a,k-n}(F^{(k)}_{-n}(x))
		\subset T_x X_k
	$$
	and
	$$
	E_{\cF^{(k)}}^s(x) := \bigcap_{n \ge 0} D F^{(k+n)}_{-n}  (F^{(k)}_{n}(x)) \; \mathcal C^s_{a,k+n}(F^{(k)}_{n}(x))
		\subset T_x X_k
	$$
	satisfy the invariance conditions
	$Df_k (x) E_{\cF^{(k)}}^*(x)= E_{\cF^{(k+1)}}^*(f_k(x))$ for $*\in \{s,u\}$.
\item there are constants $C>0$, $\delta_1>0$ and $\tilde \lambda \in (\lambda, 1)$ so that,
	for any $x\in X_k$, $k\in \mathbb Z$ and $*\in \{s,u\}$ there exists a unique smooth submanifold
	$\cW^*_{\cF^{(k)},\delta_1}(x)$ of $X_k$ (of size $\delta_1$)
	that is tangent to the subbundle $E_{\cF^{(k)}}^*(x)$ at $x$,
	in such a way that:
	\begin{itemize}
	\item[(i)] $f_k(\cW^*_{\cF^{(k)},\delta_1}(x))= \cW^*_{\cF^{(k+1)},\delta_1}(f_k(x))$
	\item[(ii)] $d_{\cW^s} ( f_k(y), f_k(z) ) \le \tilde \lambda\,  d_{\cW^s} ( y,z)$ for all
		$ y,z \in \cW^s_{\cF^{(k)},\delta_1}(x)$
	\item[(iii)] $d_{\cW^u} ( f_k^{-1}(y), f_k^{-1}(z) ) \le \tilde \lambda\, d_{\cW^u} ( y,z)$ for all
	$y,z \in \cW^u_{\cF^{(k+1)},\delta_1}(x)$
	\item[(iv)] the angles between stable and unstable bundles $E^s_{\cF^{(k)}}(x)$ and $E^u_{\cF^{(k)}}(x)$
	is bounded away from zero by some constant $\theta_k>0$
	\item[(iv)] for any $0<\vep<\delta_1$ there exists $\delta_k>0$ so that for any $x,y\in X_k$ with $d(x,y)<2\delta_k$
	the transverse intersection $\cW^s_{\cF^{(k)},\vep}(x) \pitchfork \cW^u_{\cF^{(k)},\vep}(y)$ consists of a unique point in $X_k$
	\item[(v)] $\cW^s_{\cF^{(k)},\vep}(x)=\{y\in X_k \colon d(F_n^{(k)}(y),F_n^{(k)}(x)) \le \vep \;\text{for every}\; n\ge 0\}$ and \\
	$\cW^u_{\cF^{(k)},\vep}(x)=\{y\in X_k \colon d(F_{-n}^{(k)}(y),F_{-n}^{(k)}(x)) \le \vep \;\text{for every}\; n\ge 0\}$
	\end{itemize}
\end{enumerate}
Item (1) follows from the existence of the strictly invariant cone fields. Item (2) follows from
the existence of stable and unstable manifolds for sequences of maps, using the graph transform method (cf. \cite[Section~7]{BP05}
or \cite{Acevedo}).
Here we opted to write the invariant manifolds as $\cW^*_{\cF^{(k)}}$ ($*\in\{s,u\}$) to specify the
shifted sequence $\cF^{(k)}$ with respect to which the uniform contracting or expanding behavior holds.

In the case that  $f\in \text{Diff}^{\,1}(X)$ is an Anosov diffeomorphism there exists a $C^1$-small open neighborhood $\cU$ of
$f$ so that the cone fields $\mathcal C^u_{a_0}$ and $\mathcal C^s_{a_0}$ (determined by $f$) are $Dg$-invariant
for all $g\in \cU$, and that $\lambda:=\sup_{g\in \cU} \lambda_g <1$. Furthermore,
there are constants $\theta,\delta,\vep>0$ (depending only on $f$
and $\cU$) so that any sequence $\mathcal F=\{f_n\}_{n\in \mathbb Z}$ by elements of $\cU$ is such that
$\theta< \theta_n$, $\delta < \delta_n$ and $\vep<\vep_n$ for all $n$. In other words, both the angles between stable and unstable subspaces and
the sizes given by local product structure are uniformly bounded away from zero.
In consequence, the sequence $\cF$ is $\delta_1$-expansive: if $d(F_n(x),F_n(y))<\vep$ for all $n\in \mathbb Z$ then $x\in
\cW^s_{\cF,\delta_1}(y) \cap \cW^u_{\cF,\delta_1}(y) =\{y\}$.

\subsubsection{Proof of Theorem~\ref{thm:Ainvertivel} }

Let $f$ be a $C^1$ Anosov diffeomorphism, let $\cU$ be a $C^1$-open neighborhood of $f$ as above,
and let $\mathcal F=\{f_n\}_{n\in \mathbb Z}$ be a sequence of Anosov diffeomorphisms with $f_n\in \cU$ for all $n\in \mathbb Z$.
In particular  $\lambda:=\sup_{n\in\mathbb N} \lambda_n <1$.
Let $\delta_1>0$ and $\tilde\lambda\in (\lambda,1)$ be given by items (1) and (2) above.

Fix $\beta >0$. We claim that there exists $\zeta>0$ such that every $\zeta$-pseudo orbit for $\cF$ is $\beta$-shadowed.
Take $0<\vep < (1-\tilde \lambda) \frac\beta2<\frac\beta2$ and let $0<\delta<\vep$ be given by the local product structure (cf. item (2) (iv)):
$$
d(x,y)<2\delta \quad \Rightarrow \quad \cW^s_{\cF^{(k)},\vep}(x) \pitchfork \cW^u_{\cF^{(k)},\vep}(y)
 	\;\text{consists of a unique point.}
$$
Take $N\ge 1$ such that $\tilde \lambda^N \vep <\delta\slash 2$.
Since  $\sup_{g\in \cU} \sup_{x\in X} \|Dg(x)\|<\infty$, the mean value inequality ensures the
equicontinuity of the elements of the sequence $\cF$. Thus,
there exists $0<\zeta<\delta$ such that:
\begin{enumerate}
\item[(a)]  if $(z_n)_{n\ge 0}$ is a $\zeta$-pseudo orbit and $k\ge 0$ then the finite pseudo orbit $(z_{n+k})_{n= 0}^N$
is such that
$
d(F^{(k)}_j(z_k),z_{k+j})<\frac\delta2 \quad\text{for all}\; k \in \mathbb Z \;\text{and}  \; |j|\le N
$
(cf. \cite[Theorem 3.4 ]{DR14}).
\item[(b)]
if $d(x,y)<\zeta$ then
$
d(F^{(k)}_j(x),F^{(k)}_j(y))<\frac\delta2 \quad\text{for all}\; k \in \mathbb Z \;\text{and}  \; |j|\le N
$
(by the mean value inequality).
\end{enumerate}

First we prove that the shadowing property holds for finite pseudo-orbits.
For that, we may assume without loss of generality that the pseudo-orbits $(x_n)_{n=0}^k$ are formed by
a number $k$ of points that is a multiple of $N$ (otherwise just consider the extended pseudo-orbit
$(x_n)_{n=0}^{(j+1)N}$, where $x_n=F^{(k)}_{n-k}(x_k)$ for every $k+1\le n \le (j+1)N$ and $j\in \mathbb Z_+$
is uniquely determined by $jN < k \le (j+1)N$).

Fix $p\in\mathbb N$ arbitrary and let $(x_n)_{n=0}^{pN}$ be a $\zeta$-pseudo-orbit for $\cF$.
We claim that $(x_n)_{n=0}^{pN}$ is $\beta$-shadowed by some point $x\in X$.
We define recursively $y_0=x_0$ and
\begin{equation}\label{eq:constyn}
y_{n+1} \in
	\mathcal W_{\cF^{((n+1)N)},\vep}^s(x_{(n+1)N})
	\pitchfork
	\mathcal W_{\cF^{((n+1)N)},\vep}^u(F^{(nN)}_{N}(y_{n}))
\end{equation}
for every $0\le n\le p-1$ (cf. Figure~\ref{fig44} below).
\begin{figure}[htb]\label{fig44}
\begin{center}
  \includegraphics[width=12cm,height=3.6cm]{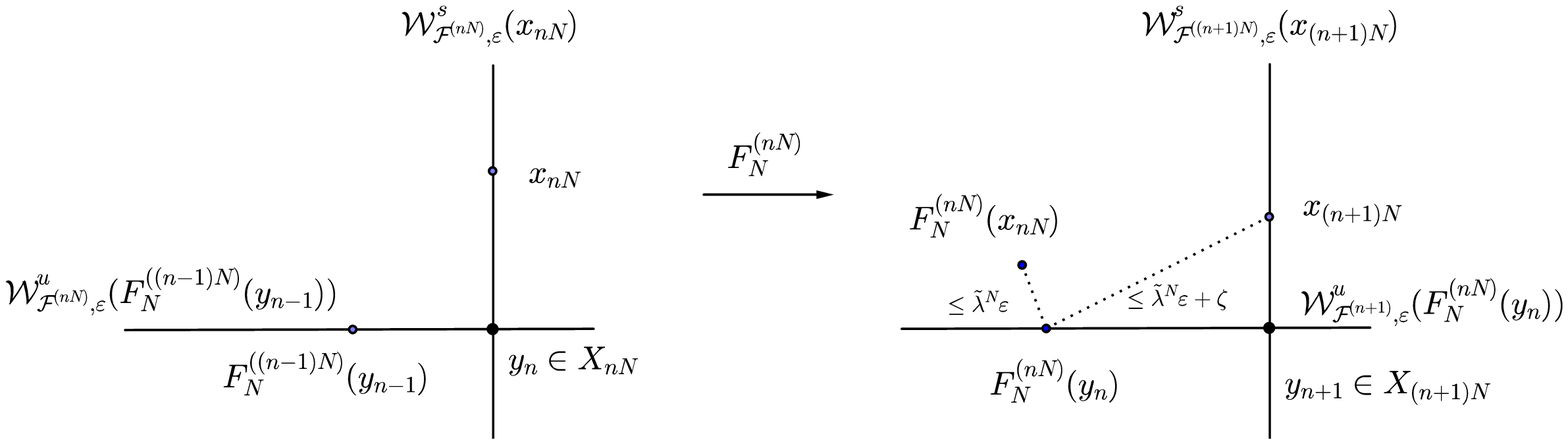}
\caption{Construction of homoclinic points}
\label{figure}
\end{center}
\end{figure}
Note that the homoclinic points $(y_n)_{n=0}^p$ as above are well defined.
Indeed, as $y_{n}\in\mathcal W_{\cF^{(nN)},\vep}^s(x_{nN})$ we have that
$$
	 d_{\cW^s}(F^{(nN)}_{N}(y_{n}), F^{(nN)}_{N}(x_{nN}))
	\le \tilde \lambda^N\, d_{\cW^s}( y_{n}, x_{nN})
	\le \tilde \lambda^N  \, \vep <\frac\delta2.
$$
In addition, as $(x_k)_k$ is a $\zeta$-pseudo-orbit,
by triangular inequality we conclude that
$$
d(x_{(n+1)N},F_{N}^{(nN)}(y_{n}))\leq d(F_{N}^{(nN)}(y_{n}),F_{N}^{(nN)}(x_{nN}))+d(F_{N}^{(nN)}(x_{nN}),x_{(n+1)N})<\delta
$$
for every $0\le n \le p$.
 This guarantees that the homoclinic point $y_{n+1}$ is well defined,
as claimed.
Now we prove that the point $x=(F_{pN})^{-1}(y_{p})=F_{-pN}^{(pN)}(y_{p}) \in X_0$ is such that its orbit $\beta$-shadows
the pseudo-orbit
$(x_n)_{n=0}^{pN}$. By construction  $y_{n}\in	\mathcal W_{\cF^{(nN)},\vep}^u(F^{((n-1)N)}_{N}(y_{n-1}))$ for all $0\le n\le p$ (recall ~\eqref{eq:constyn}). Thus, by invariance of the unstable leaves and backward contraction we get
\begin{align*}
d( F^{(nN)}_{-N} (y_{n}), y_{n-1})
            &=d((F^{(nN)}_{N})^{-1} (y_{n}),(F^{(nN)}_{N})^{-1}(F^{(nN)}_{N}(y_{n-1})))\\
            &\le \tilde \lambda^N d( y_{n}, F^{((n-1)N)}_{N}(y_{n-1}))
            \leq \tilde \lambda^{N}\vep
\end{align*}
for every $0\le n\le p$.
In particular it follows that
\begin{align*}
d(F_{nN}(x) , x_{nN} )
		& = d(F_{nN}(F_{pN}^{-1} (y_p)) , x_{nN} )
		 = d((F^{(nN)}_{(p-n)N})^{-1} (y_p) , x_{nN} ) \\
		& \le d(y_n , x_{nN} )
		+ \sum_{\ell=1}^{p-n}  \; d(\, F^{((n+\ell)N)}_{-\ell N} (y_{n+\ell}) , F^{((n+\ell-1)N)}_{-(\ell-1)N} (y_{n+\ell-1}) \,) \\
		& \le  \sum_{\ell=0}^{p-n}  \tilde \lambda^{\ell N} \vep
		\le \frac{\vep}{1-\tilde \lambda^N} < \frac{\vep}{1-\tilde \lambda} <\frac\beta2
\end{align*}
 for every $0\le n \le p$.
By equicontinuity, the choice of $0<\delta<\frac{\beta}2$ and consequence (b) of the equicontinuity
for the elements of $\cF$,
for every $0\le n \le p$ and $0\le j<N$
\begin{align*}
d(F_{j+nN}(x) , x_{j+nN} )
		& \le d(F_{j} (F_{nN}(x)) , F_j(x_{nN}) )
		 + d(F_j(x_{nN}), x_{j+nN} )
		 < \beta
\end{align*}
It follows that the sequence $(x_n)_{n=0}^{p}$ is $\beta$-shadowed
by the point $x$ and so the shadowing property follows for finite pseudo-orbits. Since $\delta>0$ is independent of $p\in\N$, a simple argument using compactness assures that any infinite $\zeta$-pseudo-orbit is $\beta$-traced by some point in $X$.
Finally, if $x$ and $y$ are two points that $\beta$-shadow  the same $\zeta=\zeta(\beta)$-pseudo orbit, then
\[
d(F_j(x),F_j(y))\leq d(F_j(x),x_j)+d(F_j(y),x_j)<2\beta<2\delta_1
\]
for all $j\in \mathbb Z$ and, by the choice of $\delta_0$, we conclude $x=y$. This proves item (1) in the theorem.
In order to deduce the Lipschitz shadowing property holds we determine a linear dependence between
the constants $\zeta$ and $\beta$. As described above, given $\beta>0$ take  $\vep=(1-\tilde \lambda)\beta/2>0$. Since the
angle between stable and unstable bundles for the sequential dynamical system is bounded away from zero and local stable and
unstable disks have bounded curvature and are tangent to the stable bundles then the local product structure property
holds: there exists $L>0$ so that if $d(x,y)<\frac{2\vep}L$ then $\cW^s_{\cF^{(k)},\vep}(x)$ and $\cW^u_{\cF^{(k)},\vep}(y)$
intersects at a unique point. Taking $\delta=\vep/L$ and $N\ge 1$ so that  $\tilde\lambda^N \vep <\delta/2$
 then we conclude that one can take $\zeta= \frac{1-\tilde \lambda}{8LM^N} \beta>0$ (with $M:=\sup_{g\in \cU} \|Dg\|_0$),
where the constant $K:=(\frac{1-\tilde \lambda}{8LM^N})^{-1}>0$ depends only on $f$.
This proves the Lipschitz shadowing property.

\medskip

We proceed to prove (2). Let $U$ is a $C^1$-small neighborhood of $f$ and $K>0$ be as above.
Assume that $|\| \cF -\cG\||<\zeta$ is small and $x\in X$.
We can find a unique point $y$ whose $\mathcal F-$orbit
$K\zeta$-shadows the $\mathcal G-$orbit of $x$. This is assured by item (1) taking into account that
$( G_j(x))_{j\in \mathbb Z}$ is a $\zeta$-pseudo-orbit for $\mathcal F$.
Thus there exists a unique point $y=:h_{\cF,\cG}(x)$ that $K\zeta$-shadows $( G_j(x))_{j\in \mathbb Z}$, i.e.,
$d(F_j(h(x)),G_j(x))<K \zeta$ for all $j\in\mathbb Z$.

\smallskip

We prove first that $h=h_{\cF,\cG}$ is an homeomorphism.
First we prove that the map $h=h_{\cF,\cG}$ is continuous.
On the one hand,
the uniqueness of the shadowing point assured by item (1) implies that
\[
h(x)=\bigcap_{j=-\infty}^{\infty}F_{j}^{-1}(B_{K\zeta}(G_j(x))),
\]
where $B_\vep(z)$ denotes the ball of radius $\vep>0$ around $z$ in $X$. In Particular, $d(h(x),x)<K\zeta$.
By hyperbolicity, the diameter of
the finite intersection $\bigcap_{j=-N}^{N}F_{j}^{-1}(B_{K\zeta}(G_j(x)))$
tends to zero (uniformly in $x$) as $N$ tends to infinity.
It implies that for $z$ sufficiently close to $x$, the points in the finite intersection
$\bigcap_{j=-N}^{N} F_{j}^{-1}(B_{K \zeta} (G_j(z))$ lie in a small ball around $h(x)$, which proves the
continuity of $h$.
To check that $h$ is injective, if $h(x_1)=h(x_2)$ then
\[
d(F_j(h(x_i)),G_j(x_i))<K \zeta \quad\text{ and }\quad F_j(h(x_1))=F_j(h(x_2))
\]
for every $j\in \mathbb Z$ and $i=1,2$.
by triangular inequality we obtain $d(G_j(x_1),G_j(x_2))<2K \zeta$, for all $j\in\mathbb Z$. If $2K \zeta$ is smaller than the
constant expansiveness of $\mathcal G$ then we conclude that $x_1=x_2$.
Now, using the fact that $h$ a continuous and one to one map on a compact and connected Riemannian manifold $X$ we get, by the invariance domain theorem, that $h$ is an open map so $h(X)$ is open. As $X$ is compact, $h(X)$ is compact and, in particular, is closed.
Altogether, this implies that $h(X)$ is a connected component of $X$, hence $h(X)=X$.
Finally, using the fact that $X$ is compact and $h$ is a continuous  one to one map on $X$ we obtain that $h$ is an homeomorphism.
The proof of \eqref{eq:conjugacyn} is entirely analogous to the one of Proposition~\ref{thm:C}.

This concludes the proof of the theorem. \hfill $\square$

\section{Ergodic stability of convergent sequences}\label{sec:cors}

\subsection{Proof of Theorem~\ref{thm:asymptinv}}

Assume that $f$ is a $C^1$-transitive Anosov diffeomorphism on a compact Riemannian manifold $X$ and let
$\mathcal F=\{f_n\}_n$ be a sequence of expanding maps so that $\lim_{n\to\infty} \|f_n-f\|_{C_1}=0$.
Proposition~\ref{thm:C} assures that there exist $L>0$ and homeomorphisms $h, h_n \in Homeo(X)$ so that
\begin{itemize}
\item[(i)] $h_n \circ f^n    = F_n \circ h$ for every $n\in \mathbb Z_+$; \vspace{.15cm}
\item[(ii)] $d_{C^0}(h_n,id) \le L |\| \cG^{(n)} - \cF^{(n)}\|| \to 0$ as $n$ tends to infinity.
\end{itemize}
Hence, for any continuous observable $\phi: X \to \mathbb R$
we get that the continuous observables $\phi_j := \phi\circ h_j$ are uniformly convergent to $\phi$ as $j\to\infty$.
Moreover, given $x\in X$,
\begin{align}
\Big| \frac1n \sum_{j=0}^{n-1} \phi (F_j( x ))
	 -  \frac1n \sum_{j=0}^{n-1}  \phi (f^j( h^{-1}(x) )) \Big|
	 & = \Big| \frac1n \sum_{j=0}^{n-1}  \phi_j (f^j( h^{-1}(x) ))
	 -  \frac1n \sum_{j=0}^{n-1}  \phi (f^j( h^{-1}(x) )) \Big|  \nonumber\\
 	& \le
  	  \frac1n \sum_{j=0}^{n-1}   \| \phi_j -\phi\|_{C^0} \label{eq:C0est}
\end{align}
which tends to zero as $n\to\infty$.
In consequence, if $\mu$ is $f$-invariant and $\phi \in C^0(X)$ then it follows from
from Birkhoff's ergodic theorem that for $(h_*\mu)$-almost every $x \in X$
 one has that $\lim_{n\to\infty} \frac1n \sum_{j=0}^{n-1} \phi (F_j( x ))$ exists
and $\int \tilde \phi \, d\mu= \int \phi \, d\mu$.
In particular if $\mu$ is $f$-invariant and ergodic then
$\lim_{n\to\infty} \frac1n \sum_{j=0}^{n-1} \delta_{F_j( x )}=\mu$
for $(h_*\mu)$-almost every $x \in X$.
This proves items (1) and (2).

We are left to prove that the Birkhoff irregular set of $\phi$ whith respect to $\cF$ has full topological entropy and
ir a Baire residual subset of $X$. Estimate \eqref{eq:C0est} implies that $x\in I_{f,\phi}$ if and only if $h^{-1}(x)\in I_{\cF,\phi}$.
In other words, $h( I_{f,\phi})=I_{\cF,\phi}$.

In particular if $\mu$ is $f$-invariant and ergodic then
$\lim_{n\to\infty} \frac1n \sum_{j=0}^{n-1} \delta_{F_j( x )}=\mu$
for $(h_*\mu)$-almost every $x \in X$.
This proves items (1) and (2).

Now assume that $\phi$ is not cohomogous to a constant with respect to $f$.
We are left to prove that the Birkhoff irregular set of $\phi$ whith respect to $\mathcal F$ has full topological entropy and
ir a Baire residual subset of $X$. Estimate \eqref{eq:C0est} implies that $x\in I_{f,\phi}$ if and only if $h^{-1}(x)\in I_{\cF,\phi}$.
In other words, $h( I_{f,\phi})=I_{\cF,\phi}$.
Using that $I_{f,\phi}$ is a Baire residual subset of $X$ (cf. \cite{BLV}) and that $h$ is a homeomorphism we conclude that $I_{\cF,\phi}$ is a Baire residual subset.

We are left to prove that $h_{I_{\cF,\phi}}(\cF)=h_{top}(\cF)$. We may assume without loss of generality that
$f$ satisfies the specification property (otherwise just take $f^k$ for some $k> 1$). In particular, $\phi$ is not cohomologous to
a constant if and only if $I_{f,\phi}\neq \emptyset$. Moreover, if $\alpha\neq \beta$ are accumulation points of the set
$\{\frac1n\sum_{j=0}^{n-1}\phi (f^j(x) \colon x\in X, \; n\ge 1\}$ and
$$
I_{f,\phi}(\alpha,\beta)= \Big\{x\in X \colon \underline{\lim}_{n\to\infty}\frac1n\sum_{j=0}^{n-1}\phi (f^j(x)=\alpha
				< \beta =\overline{\lim}_{n\to\infty}\frac1n\sum_{j=0}^{n-1}\phi (f^j(x) \Big\}
$$
then
$h_{I_{f,\phi}(\alpha,\beta)}(f) =h_{top}(f)$ (cf. \cite{Thompson}).
 We need the following:

\begin{lemma}\label{le:entropy}
Let $\cF=\{f_n\}_n$ be a sequence of $C^1$-diffeomorphisms convergent to $f$ and $Z\subset X$. Then $h_{Z}(\cF) =h_{Z}(f)$.
\end{lemma}

\begin{proof}
Note that $h_n \circ f^n    = F_n \circ h$ and that $(h_n)_n$ converges to the identity as $n\to\infty$.
For any $\vep>0$ there exists $N_\vep\ge 1$ so that $d_{C^0}(h_n,id)<\vep/3$ for all $n\ge N_\vep$. In consequence, if
$d(f^j(x),f^j(y))>\vep $ then $d(F_j (h(x)), F_j(h(y)))>\vep/3.$ This proves that if $E\subset X$ is a $(n,\vep)$-separated
set for $f$ then $h(E)$ is a $(n,\vep/3)$-separated set for $\cF$, and that
$$
s_n(\cF, \vep/3, X) \ge s_n(f, \vep, X)
$$
for every $n\ge N_\vep$. Reciprocally, if $d(F_j (x), F_j(y))>\vep$ for some $j\ge N_\vep$ then
$$
d(f^j(h^{-1}(x)),f^j(h^{-1}(y)))
	\ge d(F_j (x), F_j(y)) - \frac{2\vep}3 >\frac\vep{3}
$$
and so
$
s_j(\cF, \vep, X) \ge s_j(f, \vep/3, X).
$
The statement of the lemma is now immediate.
\end{proof}

Now, if $N\ge 1$ is large so that $\frac1n \sum_{j=0}^{n-1}   \| \phi_j -\phi\|_{C^0}< \frac{\beta-\alpha}2$ for every $n\ge N$ then
\eqref{eq:C0est} implies that
$$\underline{\lim}_{n\to\infty}\frac1n\sum_{j=0}^{n-1}\phi (F_j(x)
	< \overline{\lim}_{n\to\infty}\frac1n\sum_{j=0}^{n-1}\phi (F_j(x)
$$
for every $x\in I_{\cF,\phi,\phi}(\alpha,\beta)$. In other words, $I_{f,\phi}(\alpha,\beta)\subset I_{\cF,\phi}$ and, consequently,
$h_{I_{\cF,\phi}}(f) \ge h_{I_{f,\phi}(\alpha,\beta)}(f)= h_{top}(f)$. Finally, Lemma~\ref{le:entropy} implies that
$h_{I_{\cF,\phi}}(\cF) = h_{I_{\cF,\phi}}(f) = h_{top}(f)$. This completes the proof of the theorem.

\subsection{Proof of Theorem~\ref{thm:ASIP}}

Let $\phi: X \to \mathbb R$ be a mean zero $\alpha$-H\"older continuous and let $|\phi|_\alpha$ denote the H\"older constant.
If $\phi_j := \phi\circ h_j$ then $\|\phi_j-\phi\|_{C^0} \le |\phi|_\alpha \|h_j - id\|_{C^0}^\alpha
	\le |\phi|_\alpha L^\alpha (\sup_{\ell\ge j}\| f_\ell - f \|_{C^1})^\alpha$. In other words,
$\|\phi_j-\phi\|_{C^0} \le |\phi|_\alpha L^\alpha a_j^\alpha$. In particular
\begin{align*}
\Big|  \sum_{j=0}^{n-1} \phi \circ F_j
	 -   \sum_{j=0}^{n-1}  \phi \circ f^j \circ  h^{-1} \Big|
 	& \le  \sum_{j=0}^{n-1}   \| \phi_j -\phi\|_{C^0}
	\le |\phi|_\alpha L^\alpha \sum_{j=0}^{n-1} a_j^\alpha
\end{align*}
and, consequently,
\begin{align*}
\sum_{j=0}^{n-1} \phi \circ F_j
	 =   \sum_{j=0}^{n-1}  \phi \circ f^j \circ  h^{-1}
	 + \mathcal O( n^{\frac12-\vep})
\end{align*}
provided that $a_j \le C j^{-(\frac{1}{2}+\vep)\frac1\alpha}$ for all $j\ge 1$.
Since $\phi$ is mean zero and H\"older continuous and $f$ is uniformly expanding then
the almost sure invariance principle follows from the corresponding one for uniformly expanding
dynamics with
$\sigma^2=\int \phi^2 \, d\mu + \sum_{n= 1}^{\infty} \phi \, (\phi\circ f^j) \,d\mu>0$. (see e.g. \cite{DP}).
This proves the theorem.

\subsection*{Acknowledgments:} This work was partially supported by CNPq-Brazil.
The authors are deeply grateful to T. Bomfim for some references and useful discussions.
\medskip

\bibliographystyle{alpha}

\begin{thebibliography}{2}

\bibitem{Acevedo}
J. Acevedo,
\emph{Local stable and unstable manifolds for Anosov families},
Preprint ArXiv: 1709.00636.

\bibitem{Acevedo2}
J. Acevedo,
\emph{Structural stability of the Anosov families},
Preprint ArXiv:1709.00638.

\bibitem{Aimino}
R. Aimino, M. Nicol and S. Vaienti. \emph{Annealed and quenched limit theorems for random expanding dynamical systems},
Probability Theory and Related Fields, 162 :1 (2015) 233--274.

\bibitem{BP05}
L. Barreira and Y. Pesin,
\newblock \emph{Smooth ergodic theory and nonuniformly hyperbolic dynamics,}
\newblock Handbook of Dynamical Systems, v. 1B, Edited by B. Hasselblatt and A. Katok,
Elsevier, Amsterdam, 2005.

\bibitem{BLV}
L. Barreira, J. Li and C. Valls,
\emph{Irregular sets are residual}, Tohoku Math. J. (2)  66: 4 (2014), 471--489.

\bibitem{BB}
D. Berend, V. Bergelson,
\emph{Ergodic and mixing sequences of transformations,}
Ergod. Th. $\&$ Dynam. Sys. 4 (1984) 353--366.

\bibitem{CR}
J.-P. Conze, A. Raugi, \emph{Limit theorems for sequential expanding dynamical systems on $[0, 1]$},
Ergodic theory and related fields, 89--121, Contemp. Math., 430, Amer. Math. Soc., Providence, RI, 2007.

\bibitem{CRV1}
A. Castro, F. Rodrigues and P. Varandas,
Leafwise shadowing property for partiallly hyperbolic diffeomorphisms, Submitted 2017.

\bibitem{DP}
M. Denker and W. Philipp. \emph{Approximation by Brownian motion for Gibbs measures and flows under a function.} Ergod. Th. $\&$ Dynam. Sys. 4 (1984) 541-552.

\bibitem{DS} N. Dobbs and N. Stenlund,
\newblock \emph{Quasistatic dynamical systems},
\newblock Ergod. Th. Dynam. Sys., doi:10.1017/etds.2016.9 (2016).

\bibitem{DR14}
T. Dhaval and D. Ruchi,
\newblock  \emph{Topological stability of a sequence of maps on a compact metric space}.
\newblock  Bull. Math. Sci. (2014) 4:99--111.

 \bibitem{Hay}  S. Hayashi, \emph{Connecting invariant manifolds and the solution of the $C^1$ stability
and $\Omega$-stability conjectures for flows}, Ann. of Math. (2) 145 (1997), 1, 81--137

\bibitem{Franks74}
J. Franks,
\emph{Time dependent stable diffeomorphisms},
Inventiones Math., 24 (1974) 163--172.

\bibitem{HNTV}
N. Haydn, M. Nicol, A. Torok and S. Vaienti,
\emph{Almost sure invariance principles for sequential and non-stationary dynamical systems},
Trans. Amer. Math. Soc. (to appear)

\bibitem{KH}
A. Katok and B. Hasselblatt,
\newblock Introduction to the modern theory of dynamical systems,
\newblock Cambridge University Press, 1995.

\bibitem{Kawan}
C. Kawan. \emph{Expanding and expansive time-dependent dynamics.} Nonlinearity 28 (2015) 669--695.

\bibitem{K-S}
S. Kolyada and L. Snoha.
\newblock \emph{Topological entropy of nonautonomous
dynamical systems,}
\newblock Random and computational dynamics 4(2 and 3), 205--233 (1996).

\bibitem{Ma1} { R. Ma\~n\'e}, \emph{A proof of the $C\sp 1$ stability conjecture},
Inst. Hautes \'Etudes Sci. Publ. Math. {66} (1988) 161--210.

\bibitem{NTV}
M. Nicol, A. Torok and S. Vaienti,
\emph{Central limit theorems for sequential and random intermittent dynamical systems},
Ergod. Th. Dynam. Sys., (to appear) https://doi.org/10.1017/etds.2016.69

\bibitem{PS}
W. Philipp and W. F. Stout.
\emph{Almost Sure Invariance Principles for Partial Sums of Weakly Dependent Random Variables.}
Memoirs of the Amer. Math. Soc. 161, Amer. Math. Soc., Providence, RI, 1975.

\bibitem{S} K.~Sakai, \textit{Pseudo-orbit tracing property and
strong transversality of diffeomorphisms on closed manifolds}, Osaka J. Math.
31 (1994), no. 2, 373--386.

\bibitem{Stenlund}
M. Stenlund, \emph{Non-stationary compositions of Anosov diffeomorphisms}, Nonlinearity, 24 (2011) 2991--3018.

\bibitem{Robbin71}
J. W. Robbin.
\emph{A Structural Stability Theorem.}
Annals of Math., 94: 3 (1971), 447--493.

\bibitem{TPS}
M. Tanzi, T. Pereira and S. van Strien,
\emph{Robustness of ergodic properties of nonautonomous piecewise expanding maps},
Preprint  ArXiv:1611.04016v2

\bibitem{Thompson}
D. Thompson, {\it The irregular set for maps with the specification property has full topological pressure,}
Dyn. Syst. 25 (2010), no. 1, 25-51.

\bibitem{Viana2}
M. Viana, \emph{Stochastic dynamics of deterministic systems}, Brazillian Math. Colloquium 1997, IMPA.

\end{thebibliography}

\end{document}